\documentclass[11pt]{amsart}

\usepackage[T1]{fontenc}
\usepackage[utf8]{inputenc}
\usepackage{lmodern}
\usepackage{microtype}
\usepackage{amsmath,amssymb,amsthm,mathtools}
\usepackage{enumitem}
\usepackage{aliascnt}
\usepackage[margin=1.15in]{geometry}
\usepackage{hyperref}
\usepackage[nameinlink,capitalize]{cleveref}

\hypersetup{
  colorlinks=true,
  linkcolor=blue,
  citecolor=blue,
  urlcolor=blue,
  pdftitle={Effective Strassmann Certificates for Local p-adic Dynamical Mordell-Lang Interpolants},
  pdfauthor={Farrukh Mukhamedov},
  pdfsubject={p-adic Dynamical Mordell-Lang interpolation and Strassmann certificates},
  pdfkeywords={Dynamical Mordell-Lang, p-adic interpolation, Strassmann theorem, Mahler expansion, Tate algebra, arithmetic dynamics}
}

\title[Effective Strassmann Certificates]{Effective Strassmann Certificates for Local $p$-adic Dynamical Mordell--Lang Interpolants}
\author[F. Mukhamedov]{Farrukh Mukhamedov}
\address{Department of Mathematical Sciences, College of Science, United Arab Emirates University, P.O. Box 15551, Al Ain, Abu Dhabi, United Arab Emirates}
\email{far75m@gmail.com; farrukh.m@uaeu.ac.ae}
\date{July 7, 2026}

\newtheorem{theorem}{Theorem}[section]
\newaliascnt{proposition}{theorem}
\newtheorem{proposition}[proposition]{Proposition}
\aliascntresetthe{proposition}
\newaliascnt{lemma}{theorem}
\newtheorem{lemma}[lemma]{Lemma}
\aliascntresetthe{lemma}
\newaliascnt{corollary}{theorem}
\newtheorem{corollary}[corollary]{Corollary}
\aliascntresetthe{corollary}
\newaliascnt{definition}{theorem}
\newtheorem{definition}[definition]{Definition}
\aliascntresetthe{definition}
\newaliascnt{remark}{theorem}
\newtheorem{remark}[remark]{Remark}
\aliascntresetthe{remark}
\newaliascnt{example}{theorem}
\newtheorem{example}[example]{Example}
\aliascntresetthe{example}
\newaliascnt{fact}{theorem}
\newtheorem{fact}[fact]{Standard fact}
\aliascntresetthe{fact}

\crefname{theorem}{Theorem}{Theorems}
\Crefname{theorem}{Theorem}{Theorems}
\crefname{proposition}{Proposition}{Propositions}
\Crefname{proposition}{Proposition}{Propositions}
\crefname{lemma}{Lemma}{Lemmas}
\Crefname{lemma}{Lemma}{Lemmas}
\crefname{corollary}{Corollary}{Corollaries}
\Crefname{corollary}{Corollary}{Corollaries}
\crefname{definition}{Definition}{Definitions}
\Crefname{definition}{Definition}{Definitions}
\crefname{remark}{Remark}{Remarks}
\Crefname{remark}{Remark}{Remarks}
\crefname{example}{Example}{Examples}
\Crefname{example}{Example}{Examples}
\crefname{fact}{Standard fact}{Standard facts}
\Crefname{fact}{Standard fact}{Standard facts}

\newcommand{\Zp}{\mathbb Z_p}
\newcommand{\Qp}{\mathbb Q_p}
\newcommand{\Cp}{\mathbb C_p}
\newcommand{\OCp}{\mathcal O_{\mathbb C_p}}
\newcommand{\Np}{\mathbb Z_{\ge 0}}
\newcommand{\OK}{\mathcal O_K}
\newcommand{\vp}{v_p}
\newcommand{\ord}{\operatorname{ord}}
\newcommand{\SI}{\operatorname{SI}}
\newcommand{\A}{\mathcal A}
\newcommand{\T}{\mathcal T}
\newcommand{\id}{\mathrm{id}}
\DeclareMathOperator{\wdeg}{wdeg}

\numberwithin{equation}{section}

\begin{document}

\begin{abstract}
The $p$-adic method for Dynamical Mordell--Lang often reduces a residue class
of an orbit to the zero set of a locally analytic interpolating function.  This
paper assumes the standard interpolation, Strassmann, Mahler, and Weierstrass
tools, and studies the effective local zero-bound problem that remains after
interpolation: certifying the Strassmann index of the resulting one-variable
analytic function.  We give finite certificates for this index, including
finite-data, finite-precision, refined-tail, adaptive residue-class, one-shot,
and first-order escape criteria.  Since the Strassmann index is a rigorous
upper bound for zeros in $\Zp$ and, through Weierstrass preparation, a root
count on the closed disc over $\Cp$, these certificates give checkable stopping
criteria for local orbit-intersection computations.  A residue-class zooming
principle replaces a congruence class of times by the iterate $f^{p^h}$,
gaining $h$ additional powers of $p$ in the certificate tails.  We also
introduce an arc-ideal viewpoint for target varieties defined by several
equations, replacing a chosen hypersurface bound by the one-variable gcd of all
defining equations along the interpolated orbit.  In dimension one, the method
identifies the certified Strassmann index with the corresponding local
Weierstrass root count in the orbit ball.  Applications include certified
bounds for intersections of non-fixed power-map orbits with finite target sets,
and root-of-unity avoidance for maps tangent to the identity at torsion units.
\end{abstract}

\maketitle

\section{Introduction}

The Dynamical Mordell--Lang problem asks for the structure of the set
\[
        S(x,V)=\{n\in\Np: \varphi^n(x)\in V\}
\]
where $\varphi:X\to X$ is an algebraic self-map, $x\in X$, and $V\subseteq X$ is a subvariety.  In characteristic zero the expected structure is a finite union of arithmetic progressions.  A central method in favorable cases is $p$-adic interpolation: after passing to finitely many residue classes of the time variable, one interpolates the corresponding subsequences of the orbit by analytic functions on $\Zp$, and then applies Strassmann's theorem to the defining equations of $V$.  This is the local analytic mechanism behind the Bell--Ghioca--Tucker proof for \'{e}tale maps \cite{BGT2010}; Poonen gives a concise interpolation theorem for maps on a closed $p$-adic polydisc that are sufficiently close to the identity \cite{Poonen2014}.  For broader context on the dynamical Mordell--Lang conjecture and its $p$-adic variants, see Xie's survey \cite{Xie2023}.

This paper isolates the local zero-bound certification problem that appears after this interpolation step.  Let
\[
        \A=\Zp\langle x_1,\ldots,x_d\rangle
\]
be the Tate algebra of restricted power series over $\Zp$.  Thus $\A^d$ denotes
$d$-tuples of scalar-valued restricted power series.  We work with a map
\begin{equation}\label{eq:intro-local-map}
        f(x)=x+p^q\Phi(x),
        \qquad \Phi\in\A^d,
        \qquad q(p-1)>1,
\end{equation}
with $a\in\Zp^d$ and $F\in\Zp[x_1,\ldots,x_d]$.  For scalar-valued $G\in\A$, put
\[
        \T:\A\to\A,\qquad \T(G)=G\circ f,\qquad \Delta=\T-I.
\]
The usual $p$-adic interpolation gives
\begin{equation}\label{eq:intro-H}
        H_F(z)=\sum_{m\ge0}B_m\binom zm,
        \qquad B_m=(\Delta^mF)(a),
\end{equation}
with
\[
        H_F(n)=F(f^n(a))\qquad(n\in\Np).
\]
If $H_F\not\equiv0$, Strassmann's theorem bounds the number of zeros of $H_F$ in $\Zp$, hence the number of hitting times in this residue class.  The effective question studied here is:
\begin{quote}
\emph{Can one compute, or certify from finite data, the exact Strassmann index that controls the zeros of $H_F$?}
\end{quote}
The answer is affirmative under \eqref{eq:intro-local-map}.

The following ingredients are standard and are used here as black boxes: Mahler expansions for $p$-adic functions \cite{Mahler1958}, Strassmann's theorem \cite{Strassmann1928}, Poonen's interpolation theorem for iterates near the identity \cite{Poonen2014}, and the Weierstrass preparation and inverse-function theorems for non-archimedean analytic functions \cite{BoschGuntzerRemmert1984,Robert2000}.  The novelty of this paper is the effective layer after interpolation has been obtained.

This layer is not automatic from Strassmann's theorem alone.  Strassmann gives a zero bound once the ordinary power-series coefficients are known with enough precision, whereas the interpolants arising in the $p$-adic DML method are naturally given by Mahler coefficients, finite orbit values, or approximate $p$-adic computations.  Passing from those data to ordinary coefficients introduces factorial denominators and analytic tails.  The results below turn the analytic existence theorem into explicit finite certificates for the Strassmann index, together with residue-class refinements and multiequation gcd certificates.  They certify the index, and hence a zero bound; except in the one-dimensional orbit-ball situation, they do not assert that every zero of the interpolant is an ordinary nonnegative hitting time.

More precisely, the present paper provides:
\begin{enumerate}[label=\textup{(\arabic*)}]
\item a finite certificate for the exact Strassmann index of $H_F$ from finitely many orbit values;
\item a finite-precision version showing when approximate $p$-adic data determine the same certified index;
\item coefficientwise tail refinements using signed Stirling numbers;
\item adaptive residue-class certificates: on the class $b+p^h\Zp$ the method replaces $f$ by $f^{p^h}$ and gains $h$ additional powers of $p$ in the tail estimates;
\item an arc-ideal, or one-variable gcd, certificate for target varieties defined by several equations;
\item a one-shot upper bound from one nonzero value $F(f^{n_0}(a))$;
\item a one-dimensional contact theory identifying the certified Strassmann index with root counting in the $p$-adic orbit ball;
\item concrete arithmetic-dynamical applications to power maps on $1+p\Zp$ and to root-of-unity avoidance for maps tangent to the identity at torsion units.
\end{enumerate}
Thus the results are local and effective.  They do not produce a new global Dynamical Mordell--Lang theorem.  Rather, if a global argument has already reduced a residue class to a local map congruent to the identity, the results below give certified zero bounds and rigorous stopping criteria for the resulting analytic interpolant.

\subsection{Main results}

Throughout this paper, we use the convention
\[
        \vp(0)=+\infty.
\]

Let
\[
        D=q(p-1)-1>0
\]
and define
\begin{equation}\label{eq:intro-lambda}
        \lambda_{p,q}(M)=
        \left\lceil\frac{D(M+1)+1}{p-1}\right\rceil.
\end{equation}
Write
\[
        y_i=F(f^i(a)),
        \qquad
        B_m=\sum_{i=0}^m(-1)^{m-i}\binom mi y_i.
\]
Let $s(m,j)$ be the signed Stirling numbers of the first kind, defined by
\[
        z(z-1)\cdots(z-m+1)=\sum_{j=0}^m s(m,j)z^j.
\]
For $0\le j\le M$, define
\begin{equation}\label{eq:intro-C}
        C_j^{(M)}=
        \sum_{m=j}^M B_m\frac{s(m,j)}{m!}.
\end{equation}

\begin{theorem}\label{thm:intro-certificate}
Assume $H_F\not\equiv0$.  Put
\[
        \alpha_M=\min_{0\le j\le M}\vp(C_j^{(M)}).
\]
If
\[
        \alpha_M<\lambda_{p,q}(M),
\]
then
\[
        \SI(H_F)=
        \max\{0\le j\le M:\vp(C_j^{(M)})=\alpha_M\}.
\]
Moreover, if $H_F\not\equiv0$, this test succeeds for all sufficiently large $M$.
\end{theorem}

A finite-precision form is also available.  Let
\[
        V_M=\max_{0\le m\le M}\vp(m!).
\]
If $y_0,\ldots,y_M$ are known modulo $p^R$, then the same certificate is valid when the observed minimum is below both the analytic tail and the numerical error.

\begin{theorem}\label{thm:intro-finite-precision}
Let $\widetilde C_j^{(M)}$ be computed from arbitrary lifts of $y_0,\ldots,y_M$ known modulo $p^R$, and put
\[
        \widetilde\alpha_M=
        \min_{0\le j\le M}\vp(\widetilde C_j^{(M)}).
\]
If
\[
        \widetilde\alpha_M<
        \min\{\lambda_{p,q}(M),\,R-V_M\},
\]
then
\[
        \SI(H_F)=
        \max\{0\le j\le M:\vp(\widetilde C_j^{(M)})=\widetilde\alpha_M\}.
\]
\end{theorem}

Another contribution is an adaptive residue-class form of the certificate.  For
$h\ge0$ and $0\le b<p^h$, put
\[
        H_{F;b,h}(u)=H_F(b+p^h u).
\]
Then
\[
        H_{F;b,h}(u)=F\bigl((f^{p^h})^u(f^b(a))\bigr),
\]
and the iterate $f^{p^h}$ has the stronger congruence
\[
        f^{p^h}(x)=x+p^{q+h}\Phi_h(x),\qquad \Phi_h\in\A^d.
\]
Thus every residue class of times has its own finite certificate with $q$
replaced by $q+h$.  This gives a branch-and-bound procedure for local DML:
certify easy classes, split difficult classes into smaller congruence classes,
and record any class on which the interpolant vanishes identically as a full
arithmetic progression.

For target varieties defined by several equations, the paper also introduces
an arc-ideal certificate.  If $V$ is cut out by an ideal $I(V)$, the analytic
ideal
\[
        I_{\Gamma_a}(V)=\langle F(\Gamma_a(z)):F\in I(V)\rangle
        \subset \Qp\langle z\rangle
\]
is principal when it is nonzero.  Its generator $G_{\Gamma_a,V}$ has the same
zero set on $\Zp$ as the intersection of the interpolated orbit with $V$, so
\[
        \#\{n\in\Np:f^n(a)\in V\}\le \SI(G_{\Gamma_a,V}).
\]
This can be strictly sharper than applying Strassmann to a single defining
equation.

In dimension one the index has a geometric interpretation.

\begin{theorem}\label{thm:intro-one-dim}
Assume $d=1$, let $F\in\Zp[x]$ be nonzero, and let $a\in\Zp$.  If $f(a)=a$, then $F(f^n(a))$ is constant in $n$.  If $f(a)\ne a$, set
\[
        s=\vp(f(a)-a).
\]
Then the interpolated arc $\Gamma_a(z)=f^z(a)$ is an analytic isomorphism
\[
        \Gamma_a:\Zp\xrightarrow{\sim} a+p^s\Zp.
\]
Consequently
\[
        \SI(F(\Gamma_a(z)))
        =\SI(F(a+p^s w))
        =\sum_{\substack{\alpha\in\OCp:\,F(\alpha)=0\\
        \vp(\alpha-a)\ge s}}m_\alpha(F),
\]
where $m_\alpha(F)$ denotes root multiplicity.  At every non-fixed zero hit by the arc, the order of contact equals the multiplicity of the corresponding root of $F$.
\end{theorem}

The one-dimensional theorem gives a direct application to the arithmetic dynamical system
\[
        \varphi_r(T)=T^r
\]
on the $p$-adic unit group.  If $p\ge3$, $r\equiv1\pmod p$, and $\alpha\in1+p\Zp$, the residue disc $1+p\Zp$ is invariant under $\varphi_r$.  For a non-fixed starting point, i.e.\ assuming $\alpha^r\ne\alpha$, write
\[
        \rho=\vp(\alpha^r-\alpha).
\]
Then, for every finite target set $B\subset 1+p\Zp$, one obtains the certified bound
\[
        \#\{n\ge0: \alpha^{r^n}\in B\}
        \le
        \#\{\beta\in B: \vp(\beta-\alpha)\ge \rho\},
\]
with the evident multiplicity version.  This is proved in \cref{sec:power-map-application}; it gives a genuine
arithmetic-dynamical use of the local certificate, since once the certified
root-count bound is exhausted by a finite search, it proves the absence of
further solutions to the corresponding exponential-orbit equations.

A second application concerns torsion targets.  If $\zeta\in\Zp$ is a root of unity and $f(\zeta)\ne\zeta$ for a one-dimensional map $f(x)=x+p^q\Phi(x)$ with $q(p-1)>1$, then for every $N\ge1$ one has
\[
        \{n\in\Np:f^n(\zeta)^N=1\}
        =
        \begin{cases}
        \{0\}, & \zeta^N=1,\\[2pt]
        \varnothing, & \zeta^N\ne1.
        \end{cases}
\]
Equivalently, the local Strassmann index of $\Gamma_\zeta(z)^N-1$ is $1$ or $0$, uniformly in $N$; see \cref{sec:roots-unity-application}.

The results give the following concrete outputs.
\begin{enumerate}[label=\textup{(\arabic*)}]
\item If the certificate returns $N$, then the local equation $F(f^n(a))=0$ has at most $N$ solutions $n\in\Np$.
\item The finite-precision theorem gives a rigorous stopping rule for computations with approximate $p$-adic orbit values.
\item Residue-class zooming gives independent certificates for the subclasses $n\equiv b\pmod {p^h}$, with stronger tails as $h$ increases.
\item For multiequation targets, the arc-gcd certificate gives the Strassmann index of the one-variable generator cutting out the intersection along the analytic arc, rather than the index of a single chosen equation.
\item A single nonzero value $F(f^{n_0}(a))$ gives an immediate a priori zero bound.
\item In dimension one, the certificate is equivalent to root localization in the orbit ball $a+p^s\OCp$.
\item For power maps $T\mapsto T^r$ on $1+p\Zp$, the certificate becomes an explicit bound for intersections of the exponential orbit $\{\alpha^{r^n}:n\ge0\}$ with finite target sets.
\item For torsion units, the root-of-unity application gives a sharp uniform return bound for the whole family of divisors $x^N=1$.
\end{enumerate}

\Cref{sec:standard-facts} records the standard $p$-adic tools used without proof.  \Cref{sec:setup} fixes the local setup and constructs the interpolant.  \Cref{sec:certificate} proves the finite certificate and its termination.  \Cref{sec:precision} gives finite-precision and refined-tail versions.  \Cref{sec:zooming} introduces adaptive residue-class certificates.  \Cref{sec:arc-ideals} develops the arc-ideal and projection certificates for multiequation and finite-target problems.  \Cref{sec:bounds} gives one-shot and first-order escape criteria.  \Cref{sec:one-dimensional} develops the one-dimensional contact theory.  \Cref{sec:power-map-application} gives an arithmetic-dynamical application to power maps, and \cref{sec:roots-unity-application} gives a second application to root-of-unity avoidance.  \Cref{sec:examples} gives additional examples.  \Cref{sec:global} explains the global-to-local role and finite-extension form.  \Cref{sec:limitations} discusses limitations and further extensions.

\section{Standard \texorpdfstring{$p$}{p}-adic facts used in the note}\label{sec:standard-facts}

This section lists the background results used later.  They are standard and are cited rather than reproved.

\begin{fact}\label{fact:strassmann}
Let $K$ be a complete non-archimedean field with valuation ring $R$, and let
\[
        H(z)=\sum_{j\ge0}c_jz^j
\]
be a nonzero restricted power series with coefficients in $K$.  After scaling so that the coefficients lie in $R$ and at least one coefficient is a unit, let
\[
        N=\max\{j: |c_j|=\max_i |c_i|\}.
\]
Then $H$ has at most $N$ zeros in $R$, counted with multiplicity.  In the notation of this paper, $N=\SI(H)$.  See Stra{\ss}mann's original paper \cite{Strassmann1928}; standard textbook treatments are also given in \cite{Koblitz1984,Robert2000}.
\end{fact}

\begin{fact}\label{fact:mahler}
For $z\in\Zp$, the binomial coefficients $\binom zm$ lie in $\Zp$, and Mahler expansions express $p$-adic continuous functions on $\Zp$ in the binomial basis $\{\binom zm\}_{m\ge0}$.  The signed Stirling expansion
\[
        \binom zm=\frac1{m!}\sum_{j=0}^m s(m,j)z^j
\]
converts Mahler coefficients into ordinary power-series coefficients.  See Mahler \cite{Mahler1958} and the discussion of Mahler expansions in \cite{Robert2000}.
\end{fact}

\begin{fact}\label{fact:poonen}
Let $\A=\Zp\langle x_1,\ldots,x_d\rangle$ and let
\[
        f(x)=x+p^q\Phi(x),\qquad \Phi\in\A^d,
\]
with $q(p-1)>1$.  Then the iterates of $f$ admit a $p$-adic analytic interpolation in the time variable: there are analytic functions
\[
        G_i(x,z)\in\A\langle z\rangle\qquad(1\le i\le d)
\]
such that
\[
        G(x,n)=f^n(x)\qquad(n\in\Np).
\]
This is the local form of Poonen's interpolation theorem for maps sufficiently close to the identity \cite{Poonen2014}.  It is the same analytic input used in the $p$-adic Dynamical Mordell--Lang method, as in Bell--Ghioca--Tucker \cite{BGT2010}.
\end{fact}

\begin{fact}\label{fact:legendre}
For $m\ge1$,
\[
        \vp(m!)=\frac{m-s_p(m)}{p-1},
\]
where $s_p(m)$ is the sum of the base-$p$ digits of $m$.  This standard formula will be used only to bound denominators.
\end{fact}

\begin{fact}\label{fact:weierstrass}
Let $G(w)=\sum_{j\ge0}g_jw^j\in\Qp\langle w\rangle$ be nonzero.  Its Weierstrass degree is
\[
        \wdeg(G)=\max\{j:\vp(g_j)=\min_i\vp(g_i)\}=\SI(G).
\]
By the non-archimedean Weierstrass preparation theorem, $\wdeg(G)$ is the number of zeros of $G$ in the closed unit disc over $\Cp$, counted with multiplicity.  Moreover, analytic automorphisms of the closed unit disc preserve this degree and preserve zero divisors with multiplicity.  See \cite[Ch.~5]{BoschGuntzerRemmert1984} and \cite[Chs.~6--7]{Robert2000}.
\end{fact}

\begin{fact}\label{fact:tate-principal}
For a complete non-archimedean field $K$, every nonzero ideal of the
one-variable Tate algebra $K\langle z\rangle$ is principal.  Equivalently, any
finite family of nonzero functions in $K\langle z\rangle$ has a greatest common
divisor, unique up to multiplication by a unit.  This is a standard consequence
of Weierstrass division; see \cite[Ch.~5]{BoschGuntzerRemmert1984} and
\cite[Chs.~6--7]{Robert2000}.
\end{fact}

\begin{fact}\label{fact:inverse}
If $U(z)\in\Zp\langle z\rangle$ satisfies
\[
        U(z)\equiv uz\pmod p
\]
with $u\in\Zp^\times$, then $U$ is an analytic automorphism of the closed unit disc.  This is the one-variable $p$-adic inverse-function theorem; see \cite{BoschGuntzerRemmert1984,Robert2000}.
\end{fact}

\begin{fact}\label{fact:roots-unity-separation}
Let $\rho\in\Cp$ be a root of unity with $\rho\ne1$.  Then
\[
        \vp(\rho-1)\le \frac1{p-1}.
\]
Equivalently, if $\xi$ and $\eta$ are roots of unity and
\[
        \vp(\eta-\xi)>\frac1{p-1},
\]
then $\eta=\xi$.  This follows from the standard valuation formula for cyclotomic $p$-power roots of unity, together with the fact that roots of unity of order prime to $p$ have distinct reductions; see, for example, \cite[Ch.~3]{Koblitz1984} or \cite[Ch.~6]{Robert2000}.
\end{fact}

\section{Local setup and the interpolant}\label{sec:setup}

Fix a prime $p$ and normalize $\vp(p)=1$.  Let
\[
        \A=\Zp\langle x_1,\ldots,x_d\rangle.
\]
Thus $\A$ consists of scalar-valued restricted power series, while $\A^d$
denotes $d$-tuples of such series.  Throughout the main part of the note we assume
\begin{equation}\label{eq:q-condition}
        q\ge1,
        \qquad q(p-1)>1,
\end{equation}
and
\begin{equation}\label{eq:f-standing}
        f(x)=x+p^q\Phi(x),
        \qquad \Phi\in\A^d.
\end{equation}
Let $\T:\A\to\A$ be the composition operator
\[
        \T(G)=G\circ f\quad(G\in\A),
        \qquad \Delta=\T-I.
\]
Since $f\equiv\id\pmod{p^q}$, one has
\begin{equation}\label{eq:delta-divisibility}
        \Delta(\A)\subseteq p^q\A,
        \qquad \Delta^m(\A)\subseteq p^{qm}\A.
\end{equation}
This elementary divisibility is the source of every effective estimate below.  The next lemma records the operator calculus needed for the notation used throughout the note.  The interpolation theorem of Poonen supplies the analytic functions; the operator identities are elementary consequences of the binomial calculus and the topological nilpotence of $\Delta$.

\begin{lemma}\label{lem:analytic-powers}
For every $G\in\A$ and every $z\in\Zp$, the Mahler series
\[
        \T^zG:=(I+\Delta)^zG
        =\sum_{m\ge0}\binom zm\Delta^mG
\]
converges in $\A$.  If $q(p-1)>1$, then $z\mapsto \T^zG$ is represented by an element of $\A\langle z\rangle$.  The operators $\T^z$ satisfy
\[
        \T^{z+t}=\T^z\T^t\qquad(z,t\in\Zp),
\]
each $\T^z$ is a continuous $\Zp$-algebra automorphism of $\A$, and for $n\in\Np$ one has
\[
        \T^nG=G\circ f^n.
\]
\end{lemma}

\begin{proof}
Since $\Delta^mG\in p^{qm}\A$ and $\binom zm\in\Zp$, the defining series converges in the Banach algebra $\A$ for every fixed $z\in\Zp$.  For the assertion in $\A\langle z\rangle$, write
\[
        \binom zm=\frac1{m!}\sum_{j=0}^m s(m,j)z^j.
\]
The coefficient of $z^j$ in $\T^zG$ is
\[
        \sum_{m\ge j}\frac{s(m,j)}{m!}\Delta^mG.
\]
Its $m$th summand has valuation at least
\[
        qm-\vp(m!).
\]
Because $q(p-1)>1$, this quantity tends to $+\infty$ uniformly as $m\to\infty$, and also tends to $+\infty$ as $j\to\infty$ after taking the infimum over $m\ge j$.  Hence the ordinary power-series coefficients are well defined and tend to zero.

For $z,t\in\Zp$, Vandermonde's identity gives
\[
        \binom{z+t}{n}=\sum_{i=0}^n\binom zi\binom t{n-i}.
\]
Since $\Delta$ is topologically nilpotent, this identity may be substituted into the convergent binomial series, giving
\[
        (I+\Delta)^{z+t}=(I+\Delta)^z(I+\Delta)^t.
\]
Thus $\T^{z+t}=\T^z\T^t$.  For $n\in\Np$ the series is finite and the ordinary binomial theorem gives
\[
        \T^n=(I+\Delta)^n,
\]
which is the $n$-fold pullback by $f$.

It remains only to justify multiplicativity.  For fixed $G,H\in\A$, the function
\[
        z\longmapsto \T^z(GH)-\T^z(G)\T^z(H)
\]
is continuous on $\Zp$ and vanishes at every $n\in\Np$, because $\T^n$ is ordinary pullback.  Since $\Np$ is dense in $\Zp$, it vanishes identically.  Hence each $\T^z$ is an algebra endomorphism.  The identity $\T^z\T^{-z}=\T^0=I$ shows that it is an automorphism.
\end{proof}

For $a\in\Zp^d$ and $F\in\Zp[x_1,\ldots,x_d]$, define
\begin{equation}\label{eq:B-def}
        B_m=(\Delta^mF)(a)
\end{equation}
and
\begin{equation}\label{eq:H-def}
        H_F(z)=\sum_{m\ge0}B_m\binom zm=(\T^zF)(a).
\end{equation}
By \cref{lem:analytic-powers}, $H_F\in\Qp\langle z\rangle$ and
\begin{equation}\label{eq:H-interpolates}
        H_F(n)=F(f^n(a))\qquad(n\in\Np).
\end{equation}
Define the interpolated orbit by
\[
        \Gamma_a(z)=\bigl((\T^z x_1)(a),\ldots,(\T^z x_d)(a)\bigr).
\]
Since $\T^z$ is an algebra homomorphism, one has
\begin{equation}\label{eq:H-gamma}
        H_F(z)=F(\Gamma_a(z)).
\end{equation}

\begin{definition}\label{def:SI}
Let
\[
        H(z)=\sum_{j\ge0}c_jz^j\in\Qp\langle z\rangle
\]
be nonzero.  Set
\[
        \alpha(H)=\min_{j\ge0}\vp(c_j),
        \qquad
        \SI(H)=\max\{j\ge0:\vp(c_j)=\alpha(H)\}.
\]
\end{definition}

By \cref{fact:strassmann}, $H$ has at most $\SI(H)$ zeros in $\Zp$.

\begin{corollary}\label{cor:local-dml}
Let $V\subseteq\mathbb A^d_{\Zp}$ be cut out by $F_1,\ldots,F_r\in\Zp[x_1,\ldots,x_d]$.  If some interpolant $H_{F_i}$ is not identically zero, then
\[
        \{n\in\Np:f^n(a)\in V\}
\]
is finite.  If all $H_{F_i}$ are identically zero, then every iterate lies on $V$.
\end{corollary}

\begin{proof}
If $H_{F_i}\not\equiv0$, then \cref{fact:strassmann} gives finitely many zeros of $H_{F_i}$ in $\Zp$, hence finitely many $n\in\Np$ with $F_i(f^n(a))=0$.  The hitting set for $V$ is a subset of this finite set.  If all interpolants vanish identically, then all defining equations vanish at every iterate by \eqref{eq:H-interpolates}.
\end{proof}

\begin{proposition}\label{prop:finite-difference}
Let
\[
        y_i=F(f^i(a))\qquad(i\ge0).
\]
Then
\begin{equation}\label{eq:B-finite-difference}
        B_m=\sum_{i=0}^m(-1)^{m-i}\binom mi y_i.
\end{equation}
\end{proposition}

\begin{proof}
Since $\Delta=\T-I$,
\[
        \Delta^m=(\T-I)^m=
        \sum_{i=0}^m(-1)^{m-i}\binom mi\T^i.
\]
Evaluating at $a$ gives \eqref{eq:B-finite-difference}.
\end{proof}

Using \cref{fact:mahler}, write
\[
        H_F(z)=\sum_{j\ge0}c_jz^j.
\]
Then
\begin{equation}\label{eq:cj-formula}
        c_j=\sum_{m\ge j}B_m\frac{s(m,j)}{m!}.
\end{equation}
For $M\ge0$ and $0\le j\le M$, define the truncated coefficient
\begin{equation}\label{eq:Cjm-def}
        C_j^{(M)}=\sum_{m=j}^M B_m\frac{s(m,j)}{m!}.
\end{equation}
Finally, set
\begin{equation}\label{eq:D-def}
        D=q(p-1)-1>0,
\end{equation}
and
\begin{equation}\label{eq:lambda-def}
        \lambda_{p,q}(M)=
        \left\lceil\frac{D(M+1)+1}{p-1}\right\rceil.
\end{equation}

\section{Exact finite certificates}\label{sec:certificate}

\begin{lemma}\label{lem:tail}
Let $H_F(z)=\sum c_jz^j$.  For $0\le j\le M$,
\[
        \vp(c_j-C_j^{(M)})\ge\lambda_{p,q}(M),
\]
and for $j>M$,
\[
        \vp(c_j)\ge\lambda_{p,q}(M).
\]
\end{lemma}

\begin{proof}
By \eqref{eq:delta-divisibility}, $B_m\in p^{qm}\Zp$.  Since $s(m,j)\in\mathbb Z$, every term
\[
        B_m\frac{s(m,j)}{m!}
\]
has valuation at least $qm-\vp(m!)$.  By \cref{fact:legendre}, for $m\ge1$,
\[
        qm-\vp(m!)
        =\frac{(q(p-1)-1)m+s_p(m)}{p-1}
        \ge\frac{Dm+1}{p-1}.
\]
If $m\ge M+1$, this is at least $\lambda_{p,q}(M)$.  The claimed bounds follow from \eqref{eq:cj-formula} and \eqref{eq:Cjm-def}.
\end{proof}

\begin{theorem}\label{thm:certificate}
Assume $H_F\not\equiv0$.  Fix $M\ge0$ and put
\[
        \alpha_M=\min_{0\le j\le M}\vp(C_j^{(M)}).
\]
If
\begin{equation}\label{eq:certificate-condition}
        \alpha_M<\lambda_{p,q}(M),
\end{equation}
then
\begin{equation}\label{eq:certificate-output}
        \SI(H_F)=
        \max\{0\le j\le M:\vp(C_j^{(M)})=\alpha_M\}.
\end{equation}
\end{theorem}

\begin{proof}
Let
\[
        J_M=\max\{0\le j\le M:\vp(C_j^{(M)})=\alpha_M\}.
\]
For $j\le M$, \cref{lem:tail} gives
\[
        \vp(c_j-C_j^{(M)})\ge\lambda_{p,q}(M)>\alpha_M.
\]
Thus the coefficients $c_j$ and $C_j^{(M)}$ have the same valuation whenever either has valuation $\alpha_M$, and no coefficient with $j\le M$ has valuation below $\alpha_M$.  For $j>M$, \cref{lem:tail} gives $\vp(c_j)>\alpha_M$.  Hence the least valuation among all ordinary coefficients of $H_F$ is $\alpha_M$, and the largest index where it occurs is $J_M$.
\end{proof}

\begin{corollary}\label{cor:termination}
If $H_F\not\equiv0$, then \eqref{eq:certificate-condition} holds for all sufficiently large $M$.  Therefore the algorithm that increases $M$ until the certificate succeeds terminates and returns $\SI(H_F)$.
\end{corollary}

\begin{proof}
Let
\[
        \alpha=\min_j\vp(c_j),
        \qquad J=\SI(H_F).
\]
Choose $M\ge J$ with $\lambda_{p,q}(M)>\alpha$.  By \cref{lem:tail}, $C_j^{(M)}$ differs from $c_j$ by a term of valuation greater than $\alpha$ for every $j\le M$.  Thus the minimum of the valuations of $C_j^{(M)}$ is $\alpha$, attained at $j=J$, and the certificate succeeds.
\end{proof}

\begin{remark}
If the certificate does not succeed, then either $M$ is too small or $H_F\equiv0$.  In the latter case $F(f^n(a))=0$ for all $n\in\Np$, so the local DML contribution is the full progression represented by this residue class.
\end{remark}

\section{Finite precision and sharper tails}\label{sec:precision}

\begin{theorem}\label{thm:finite-precision}
Fix $M\ge0$ and suppose that the orbit values $y_0,\ldots,y_M$ are known modulo $p^R$.  Choose arbitrary lifts and compute approximations $\widetilde B_m$ and $\widetilde C_j^{(M)}$ using \eqref{eq:B-finite-difference} and \eqref{eq:Cjm-def}.  Let
\[
        V_M=\max_{0\le m\le M}\vp(m!),
        \qquad E=R-V_M,
\]
and put
\[
        \widetilde\alpha_M=
        \min_{0\le j\le M}\vp(\widetilde C_j^{(M)}).
\]
If
\begin{equation}\label{eq:finite-precision-condition}
        \widetilde\alpha_M<
        \min\{\lambda_{p,q}(M),E\},
\end{equation}
then
\[
        \SI(H_F)=
        \max\{0\le j\le M:\vp(\widetilde C_j^{(M)})=\widetilde\alpha_M\}.
\]
\end{theorem}

\begin{proof}
The finite-difference formula uses only integral binomial coefficients, so $\widetilde B_m-B_m\in p^R\Zp$ for $m\le M$.  In forming $C_j^{(M)}$, the largest possible loss of precision comes from division by factorials of integers at most $M$.  Hence
\[
        \widetilde C_j^{(M)}-C_j^{(M)}\in p^{R-V_M}\Zp=p^E\Zp.
\]
If \eqref{eq:finite-precision-condition} holds, then the observed minimum is below the numerical error and below the analytic tail.  Therefore it is the true minimum for the exact truncated coefficients, and \cref{thm:certificate} applies.
\end{proof}

\begin{center}
\fbox{%
\begin{minipage}{0.92\textwidth}
\textbf{Certified Strassmann-index algorithm.}
Given $p,q,f,F,a,M$ and $p$-adic precision $R$:
\begin{enumerate}[label=\textup{\arabic*.},leftmargin=2.2em]
\item Compute $y_i=F(f^i(a))$ for $0\le i\le M$ modulo $p^R$.
\item Compute the finite differences $\widetilde B_m$ from \eqref{eq:B-finite-difference}.
\item Compute the truncated ordinary coefficients $\widetilde C_j^{(M)}$ from \eqref{eq:Cjm-def}.
\item Set $V_M=\max_{0\le m\le M}\vp(m!)$ and
\[
        \widetilde\alpha_M=\min_{0\le j\le M}\vp(\widetilde C_j^{(M)}).
\]
\item If
\[
        \widetilde\alpha_M<\min\{\lambda_{p,q}(M),R-V_M\},
\]
return
\[
        \max\{0\le j\le M:\vp(\widetilde C_j^{(M)})=\widetilde\alpha_M\}.
\]
Otherwise return \emph{inconclusive} and increase $M$ or $R$.
\end{enumerate}
The correctness of the successful output is exactly \cref{thm:finite-precision}; inconclusive output means that the chosen truncation or precision was insufficient, or that the interpolant may be identically zero.  A separate zero-detection step is needed to distinguish these possibilities.
\end{minipage}}
\end{center}

\begin{proposition}\label{prop:sharp-tail}
Fix $M\ge0$.  Suppose we are given numbers
\[
        L_0,\ldots,L_M,L_\infty
\]
satisfying
\[
        \vp(c_j-C_j^{(M)})\ge L_j\quad(0\le j\le M),
        \qquad
        \vp(c_j)\ge L_\infty\quad(j>M).
\]
If
\[
        \alpha_M<\min\{L_0,\ldots,L_M,L_\infty\},
\]
then
\[
        \SI(H_F)=
        \max\{0\le j\le M:\vp(C_j^{(M)})=\alpha_M\}.
\]
A valid choice is
\[
        L_j=\inf_{m\ge M+1}
        \left(qm+\vp(s(m,j))-\vp(m!)\right),
        \qquad 0\le j\le M,
\]
and $L_\infty=\lambda_{p,q}(M)$, with $\vp(0)=+\infty$.
\end{proposition}

\begin{proof}
The proof of \cref{thm:certificate} uses only lower bounds for the tails.  Replacing the uniform tail bound by the displayed bounds gives the first assertion.  The stated choice of $L_j$ follows directly from
\[
        c_j-C_j^{(M)}=
        \sum_{m\ge M+1}B_m\frac{s(m,j)}{m!},
        \qquad B_m\in p^{qm}\Zp.
\]
For $j>M$, \cref{lem:tail} supplies $L_\infty$.
\end{proof}

\section{Adaptive residue-class certificates}\label{sec:zooming}

The certificate in \cref{sec:certificate,sec:precision} treats the whole time
disc $\Zp$ at once.  In computations, and also in local DML arguments, it is
useful to zoom in on a congruence class of times.  This has an intrinsic
dynamical meaning: on the class $b+p^h\Zp$ one starts at $f^b(a)$ and iterates
$f^{p^h}$.  The key point is that this iterate is closer to the identity than
$f$ itself.

\begin{lemma}\label{lem:iterate-gains-precision}
For every integer $h\ge0$ there exists $\Phi_h\in\A^d$ such that
\[
        f^{p^h}(x)=x+p^{q+h}\Phi_h(x).
\]
\end{lemma}

\begin{proof}
Let $x_i$ be a coordinate function.  Since the pullback by $f$ is $\T=I+\Delta$,
\[
        (f^{p^h})_i-x_i=(\T^{p^h}-I)x_i
        =\sum_{m=1}^{p^h}\binom {p^h}m\Delta^m x_i.
\]
The term with index $m$ lies in
\[
        p^{qm+\vp\binom {p^h}m}\A.
\]
Using
\[
        \binom {p^h}m=\frac{p^h}{m}\binom{p^h-1}{m-1},
\]
we have $\vp\binom {p^h}m\ge h-\vp(m)$.  Therefore the $m$th term has valuation
at least
\[
        qm+h-\vp(m)=q+h+\bigl(q(m-1)-\vp(m)\bigr).
\]
Since $q\ge1$ and $\vp(m)\le m-1$, this is at least $q+h$.  Hence each coordinate
of $f^{p^h}(x)-x$ lies in $p^{q+h}\A$.
\end{proof}

Fix $h\ge0$ and, if $h>0$, a representative $0\le b<p^h$; for $h=0$ take
$b=0$.  Put
\[
        f_h=f^{p^h},
        \qquad a_b=f^b(a),
\]
and define the zoomed interpolant
\begin{equation}\label{eq:zoomed-interpolant}
        H_{F;b,h}(u)=F(f_h^u(a_b)).
\end{equation}

\begin{proposition}\label{prop:zoomed-interpolant}
For all $u\in\Zp$,
\begin{equation}\label{eq:zoomed-identity}
        H_{F;b,h}(u)=H_F(b+p^h u).
\end{equation}
Moreover, if
\[
        y_i^{(b,h)}=F(f^{b+p^h i}(a))\qquad(i\ge0),
\]
then the Mahler coefficients of $H_{F;b,h}$ are the finite differences
\[
        B_m^{(b,h)}=
        \sum_{i=0}^m(-1)^{m-i}\binom mi y_i^{(b,h)}.
\]
Consequently, whenever $H_{F;b,h}\not\equiv0$, the exact, finite-precision, and
refined-tail certificates of \cref{sec:certificate,sec:precision} apply to
$H_{F;b,h}$ with $q$ replaced by $q+h$.
\end{proposition}

\begin{proof}
For every $n\in\Np$,
\[
        H_{F;b,h}(n)=F(f_h^n(a_b))=F(f^{b+p^h n}(a))=H_F(b+p^h n).
\]
Both sides are analytic functions of $u\in\Zp$, and $\Np$ is dense in $\Zp$; hence
\eqref{eq:zoomed-identity} follows.  The finite-difference formula is
\cref{prop:finite-difference} applied to the local system $(f_h,a_b)$.  By
\cref{lem:iterate-gains-precision}, this local system satisfies the standing
hypothesis with $q+h$ in place of $q$, so the certificates apply with the tail
function $\lambda_{p,q+h}$.
\end{proof}

\begin{theorem}\label{thm:residue-class-certificate}
Fix $h\ge0$ and representatives $b$ modulo $p^h$.  For each $b$, let
\[
        S_{b,h}=\{n\in\Np:n\equiv b\pmod {p^h},\ F(f^n(a))=0\}.
\]
If $H_{F;b,h}\not\equiv0$, then
\[
        \#S_{b,h}\le \SI(H_{F;b,h}),
\]
and this index is certifiable from finitely many values
$F(f^{b+p^h i}(a))$.  If $H_{F;b,h}\equiv0$, then the whole progression
$b+p^h\Np$ lies in the local zero set.  In particular, if no class is
identically zero, then
\[
        \#\{n\in\Np:F(f^n(a))=0\}
        \le\sum_{b\bmod p^h}\SI(H_{F;b,h}).
\]
\end{theorem}

\begin{proof}
The congruence class $n\equiv b\pmod {p^h}$ is represented uniquely as
$n=b+p^h i$ with $i\in\Np$.  By \eqref{eq:zoomed-identity}, zeros in this class
are exactly ordinary zeros of $H_{F;b,h}$ at nonnegative integer arguments.
If the zoomed interpolant is nonzero, Strassmann's theorem gives the stated
bound.  If it is identically zero, then $F(f^{b+p^h i}(a))=0$ for every
$i\in\Np$.
\end{proof}

\begin{center}
\fbox{%
\begin{minipage}{0.92\textwidth}
\textbf{Adaptive zooming strategy.}
Start with the single class $(b,h)=(0,0)$.
For a class $(b,h)$, compute the certificate for $H_{F;b,h}$ using the tail
$\lambda_{p,q+h}$.  If it succeeds, store the certified local bound.  If the
interpolant is proved to vanish identically, store the full progression
$b+p^h\Np$.  If neither conclusion is available, split the class into the $p$
children
\[
        b+kp^h+p^{h+1}\Zp\qquad(0\le k<p),
\]
and repeat.  The gain from splitting is twofold: the analytic tail improves
from $q$ to $q+h+1$, and many child classes may have no zeros or much smaller
Strassmann index.
\end{minipage}}
\end{center}

\begin{remark}
Residue-class zooming does not change the underlying orbit; it changes the
coordinate on the time disc.  Thus it is compatible with the global DML
procedure of passing to residue classes.  The theorem above says that one can
continue this passage inside the local analytic problem itself, with a
quantitative gain in the certificate estimates at every $p$-power refinement.
\end{remark}

\section{Arc ideals and multiequation certificates}\label{sec:arc-ideals}

The preceding sections certify a single equation $F(\Gamma_a(z))=0$.  If a
target variety is defined by several equations, applying Strassmann to one
chosen equation can be wasteful.  The natural one-variable object is instead
the ideal cut out by all equations along the analytic arc.

Let $V\subseteq\mathbb A^d_{\Zp}$ be cut out by an ideal
$I(V)\subseteq\Zp[x_1,\ldots,x_d]$.  Define the arc ideal
\begin{equation}\label{eq:arc-ideal}
        I_{\Gamma_a}(V)=
        \left\langle F(\Gamma_a(z)):F\in I(V)\right\rangle
        \subseteq \Qp\langle z\rangle.
\end{equation}
If this ideal is nonzero, \cref{fact:tate-principal} gives a generator, unique
up to a unit.  We denote any such generator by $G_{\Gamma_a,V}$ and call it the
arc-gcd of $V$ along the interpolated orbit.

\begin{proposition}\label{prop:arc-gcd-certificate}
With notation as above, exactly one of the following alternatives holds.
\begin{enumerate}[label=\textup{(\alph*)}]
\item $I_{\Gamma_a}(V)=(0)$.  Then $\Gamma_a(\Zp)\subseteq V$, and in particular
$f^n(a)\in V$ for every $n\in\Np$.
\item $I_{\Gamma_a}(V)\ne(0)$.  Then, for every $t\in\Zp$,
\[
        \Gamma_a(t)\in V
        \quad\Longleftrightarrow\quad
        G_{\Gamma_a,V}(t)=0.
\]
Consequently
\[
        \#\{n\in\Np:f^n(a)\in V\}
        \le \SI(G_{\Gamma_a,V}).
\]
Moreover, if $H_F=F(\Gamma_a(z))$ is any nonzero element of
$I_{\Gamma_a}(V)$, then
\[
        \SI(G_{\Gamma_a,V})\le \SI(H_F).
\]
\end{enumerate}
\end{proposition}

\begin{proof}
If the arc ideal is zero, every polynomial in $I(V)$ vanishes on $\Gamma_a(z)$;
therefore the analytic arc, and hence every ordinary iterate, lies on $V$.
Assume the ideal is nonzero and generated by $G_{\Gamma_a,V}$.  If
$\Gamma_a(t)\in V$, then every element of the arc ideal vanishes at $t$, so in
particular $G_{\Gamma_a,V}(t)=0$.  Conversely, if $G_{\Gamma_a,V}(t)=0$, then
every element of the arc ideal is a multiple of $G_{\Gamma_a,V}$ and therefore
vanishes at $t$; hence all equations in $I(V)$ vanish at $\Gamma_a(t)$.
Strassmann's theorem gives the zero bound on ordinary times.

Finally, if $H_F\in I_{\Gamma_a}(V)$ is nonzero, then
$H_F=G_{\Gamma_a,V}Q$ for some $Q\in\Qp\langle z\rangle$.  Under the
Weierstrass interpretation of \cref{fact:weierstrass}, the degree of the zero
divisor of a product is the sum of the degrees of the zero divisors of the
factors.  Hence the Strassmann index of the common divisor cannot exceed that
of $H_F$.
\end{proof}

\begin{remark}\label{rem:arc-gcd-sharpness}
The gain can be strict.  For example, if two defining equations restrict to
$H_1(z)=z(z-1)$ and $H_2(z)=z(z-2)$ along the same arc, then each single
equation gives a Strassmann bound $2$, while the arc-gcd is $z$ and the common
zero bound is $1$.  Thus the arc-gcd records the intersection with the target
rather than the zero set of an auxiliary hypersurface.
\end{remark}

The arc-gcd is an exact object.  The result above is therefore a correctness
theorem: once a generator, or a certified gcd of the pulled-back equations, has
been obtained, its Strassmann index gives the common-zero bound.  The present
paper does not construct a full finite-precision Weierstrass-gcd algorithm.  In
computation one may approximate the generator by a finite-precision
Weierstrass gcd of a finite set of generators of $I(V)$ and then apply the
coefficient certificate to the resulting one-variable function.  The
finite-precision theorem remains the verification step: a proposed gcd is useful
only after its Strassmann index is certified below both the analytic tail and the
numerical error.

A simpler but often effective screening device is to project a finite target
set to one coordinate or to a generic linear form.

\begin{proposition}\label{prop:projection-certificate}
Let $B\subset\Zp^d$ be finite and let
\[
        L(x)=u_1x_1+\cdots+u_dx_d,
        \qquad u_i\in\Zp,
\]
be an integral linear form.  Put
\[
        P_{B,L}(x)=\prod_{\beta\in B}\bigl(L(x)-L(\beta)\bigr)
        \in\Zp[x_1,\ldots,x_d]
\]
and
\[
        H_{B,L}(z)=P_{B,L}(\Gamma_a(z)).
\]
If $H_{B,L}\not\equiv0$, then
\[
        \#\{n\in\Np:f^n(a)\in B\}
        \le \SI(H_{B,L}).
\]
Thus $\SI(H_{B,L})=0$ certifies avoidance of $B$.  More generally, if a finite
search has found $\SI(H_{B,L})$ distinct ordinary zeros of $H_{B,L}$, then no
further iterate can lie in $B$.
\end{proposition}

\begin{proof}
If $f^n(a)\in B$, then $P_{B,L}(f^n(a))=0$, so every target hit is a zero of
$H_{B,L}$ at the corresponding time.  Strassmann's theorem gives the bound.
If all zeros allowed by the Strassmann index have already been found, then
there are no further zeros of $H_{B,L}$ in $\Zp$, hence no further target hits.
\end{proof}

\begin{remark}
The bad linear forms for which two distinct points of $B$ have the same
projection lie in a finite union of proper hyperplanes in the coefficient space
of $L$.  Thus a generic integral linear form separates the target values.  Even
when the projection has extra zeros not corresponding to target hits, the
certificate remains valid because it bounds a larger zero set.
\end{remark}

\section{One-shot bounds and first-order escape}\label{sec:bounds}

\begin{proposition}\label{prop:one-shot}
Assume $H_F\not\equiv0$.  Suppose that for some $n_0\in\Np$,
\[
        H_F(n_0)=F(f^{n_0}(a))\ne0,
\]
and set
\[
        \beta=\vp(H_F(n_0)).
\]
Then
\begin{equation}\label{eq:one-shot-bound}
        \SI(H_F)\le
        N_{p,q}(\beta):=
        \max\left\{0,
        \left\lfloor\frac{(p-1)\beta-1}{q(p-1)-1}\right\rfloor
        \right\}.
\end{equation}
\end{proposition}

\begin{proof}
Let $H_F(z)=\sum c_jz^j$ and set $\alpha=\min_j\vp(c_j)$.  Since $n_0\in\Zp$, evaluation at $n_0$ cannot decrease the minimum coefficient valuation, so $\alpha\le\beta$.  For $j\ge1$, \eqref{eq:cj-formula} and \cref{fact:legendre} give
\[
        \vp(c_j)\ge\frac{(q(p-1)-1)j+1}{p-1}.
\]
If this lower bound is greater than $\beta$, then $j$ cannot be an index at which the minimum coefficient valuation is attained.  Solving the resulting inequality gives \eqref{eq:one-shot-bound}.
\end{proof}

\begin{remark}
For $p=3$ and $q=1$, the bound is $N_{3,1}(\beta)=\max\{0,2\beta-1\}$.  The index bounds the number of zeros in $\Zp$, not the largest ordinary zero.  For instance, $z-p^L$ has one zero in $\Zp$, but that zero is the ordinary integer $p^L$.
\end{remark}

\begin{proposition}\label{prop:first-order-escape}
Let $F\in\Zp[x_1,\ldots,x_d]$ and $a\in\Zp^d$.  Assume
\[
        F(a)=0
\]
and
\begin{equation}\label{eq:transversality}
        L:=\sum_{i=1}^d
        \frac{\partial F}{\partial x_i}(a)\Phi_i(a)
        \in \Zp^\times.
\end{equation}
Then $\SI(H_F)=1$.  Hence $F(f^n(a))\ne0$ for every $n\ge1$.
\end{proposition}

\begin{proof}
Taylor expansion gives
\[
        F(a+p^q\Phi(a))-F(a)=p^qL+O(p^{2q}).
\]
Thus $B_1=(\Delta F)(a)$ has valuation $q$, while $B_0=0$.  For $m\ge2$, the contribution of the $m$th Mahler coefficient to the ordinary coefficient of $z$ has valuation at least
\[
        qm-\vp(m!)=q+\bigl(q(m-1)-\vp(m!)\bigr).
\]
By \cref{fact:legendre},
\[
        q(m-1)-\vp(m!)
        \ge \left(q-\frac1{p-1}\right)(m-1)>0,
\]
and the left-hand side is an integer.  Hence the displayed valuation is greater than $q$, and the ordinary coefficient $c_1$ has valuation $q$.

For $j\ge2$, every term contributing to $c_j$ has $m\ge j\ge2$, and its valuation is at least
\[
        qm-\vp(m!)=q+\bigl(q(m-1)-\vp(m!)\bigr)>q
\]
by the same estimate.  Since $c_0=H_F(0)=F(a)=0$, the least coefficient valuation is attained only at $j=1$, so $\SI(H_F)=1$.  Since $H_F(0)=0$, Strassmann's theorem implies that there are no further zeros.
\end{proof}

\section{The one-dimensional contact theory}\label{sec:one-dimensional}

In this section $d=1$ and
\[
        f(x)=x+p^q\Phi(x),
        \qquad \Phi\in\Zp\langle x\rangle,
        \qquad q(p-1)>1.
\]
Set
\[
        g(x)=f(x)-x=p^q\Phi(x).
\]
For $a\in\Zp$, write
\[
        \Gamma_a(z)=f^z(a).
\]
This analytic arc is supplied by \cref{fact:poonen}.

\begin{lemma}\label{lem:one-delta-factor}
For every $m\ge1$ there exists $Q_m\in\Zp\langle x\rangle$ such that
\[
        \Delta^m x=g(x)Q_m(x),
        \qquad
        Q_m\in p^{q(m-1)}\Zp\langle x\rangle.
\]
Consequently, if $\delta=f(a)-a$, then
\[
        (\Delta^m x)(a)\in \delta\,p^{q(m-1)}\Zp
        \qquad(m\ge1).
\]
\end{lemma}

\begin{proof}
For any $H\in\Zp\langle x\rangle$, write $H(T)=\sum_{n\ge0}h_nT^n$.  Then
\[
        H(Y)-H(X)=(Y-X)Q_H(X,Y),
\]
where
\[
        Q_H(X,Y)=
        \sum_{n\ge1}h_n\sum_{i=0}^{n-1}Y^iX^{n-1-i}
        \in\Zp\langle X,Y\rangle.
\]
The assertion that $Q_H$ is restricted follows because the coefficient of $X^rY^s$ is $h_{r+s+1}$, which tends to $0$ $p$-adically as $r+s\to\infty$.  Substituting $Y=f(x)$ and $X=x$ gives
\[
        H(f(x))-H(x)=(f(x)-x)Q_H(x,f(x)),
\]
so the difference is divisible by $f(x)-x=g(x)$ in the Tate algebra.  For $m=1$, $\Delta x=g$, so $Q_1=1$.

Assume $\Delta^m x=gQ_m$ with $Q_m=p^{q(m-1)}U_m$.  Since $g=p^q\Phi$, the preceding divisibility applied to $H=\Phi$ gives $\Phi(f)-\Phi=gR$ for some $R\in\Zp\langle x\rangle$, and hence
\[
        g(f(x))=p^q\Phi(f(x))=g(x)+p^qg(x)R(x)=g(x)(1+p^qR(x)).
\]
Therefore
\begin{align*}
        \Delta^{m+1}x
        &=\Delta(gQ_m) \\
        &=g(f)Q_m(f)-gQ_m \\
        &=g\,p^{q(m-1)}\bigl(U_m(f)-U_m+p^qR\,U_m(f)\bigr).
\end{align*}
Since $U_m(f)-U_m\in p^q\Zp\langle x\rangle$, the bracket lies in $p^q\Zp\langle x\rangle$, proving the induction step.
\end{proof}

\begin{lemma}\label{lem:orbit-denominator}
For every $m\ge2$ one has
\[
        q(m-1)-\vp(m!)\ge1.
\]
\end{lemma}

\begin{proof}
By \cref{fact:legendre},
\[
        \vp(m!)=\frac{m-s_p(m)}{p-1}\le\frac{m-1}{p-1}.
\]
Hence
\[
        q(m-1)-\vp(m!)
        \ge\left(q-\frac1{p-1}\right)(m-1)>0.
\]
The left-hand side is an integer, so it is at least $1$.
\end{proof}

\begin{theorem}\label{thm:orbit-ball}
Let $a\in\Zp$.  If $f(a)=a$, then $\Gamma_a(z)\equiv a$.  If $f(a)\ne a$, set
\[
        \delta=f(a)-a,
        \qquad s=\vp(\delta).
\]
Then there exists $W_a(z)\in\Zp\langle z\rangle$ such that
\[
        \Gamma_a(z)=a+\delta W_a(z),
        \qquad W_a(z)=z+p\Theta_a(z)
\]
with $\Theta_a\in\Zp\langle z\rangle$.  Hence $W_a$ is an analytic automorphism of the closed unit disc and
\[
        \Gamma_a:\Zp\xrightarrow{\sim}a+\delta\Zp=a+p^s\Zp.
\]
\end{theorem}

\begin{proof}
If $f(a)=a$, then every ordinary iterate is equal to $a$; by interpolation, the analytic arc is constant.  Assume $\delta\ne0$.  By \cref{lem:one-delta-factor},
\[
        (\Delta^m x)(a)=\delta Q_m(a),
        \qquad Q_m(a)\in p^{q(m-1)}\Zp,
\]
with $Q_1(a)=1$.  Hence
\[
        \Gamma_a(z)=a+
        \delta\left(z+
        \sum_{m\ge2}Q_m(a)\binom zm\right).
\]
Set
\[
        W_a(z)=z+\sum_{m\ge2}Q_m(a)\binom zm.
\]
After expanding the Mahler polynomials in ordinary powers of $z$, the contribution of the $m$th summand to any ordinary coefficient has valuation at least
\[
        q(m-1)-\vp(m!).
\]
By \cref{lem:orbit-denominator}, this is at least $1$ for every $m\ge2$, so all coefficients of the tail lie in $p\Zp$.  Moreover, for ordinary degree $j$, the same estimate gives the lower bound
\[
        \inf_{m\ge\max\{2,j\}}\{q(m-1)-\vp(m!)\},
\]
which tends to $+\infty$ as $j\to\infty$.  Therefore the tail is an element of $p\Zp\langle z\rangle$, so $W_a(z)=z+p\Theta_a(z)$ with $\Theta_a\in\Zp\langle z\rangle$.  The analytic automorphism assertion follows from \cref{fact:inverse}.
\end{proof}

\begin{lemma}\label{lem:flow-law}
For $a\in\Zp$ and $z,t\in\Zp$,
\[
        \Gamma_a(z+t)=\Gamma_{\Gamma_a(z)}(t).
\]
\end{lemma}

\begin{proof}
This is the operator identity $\T^{z+t}=\T^z\T^t$ from \cref{lem:analytic-powers}, evaluated at the coordinate function and then at $a$.
\end{proof}

\subsection{Contact order}

Let $F\in\Zp[x]$ be nonzero.  For $b\in\Zp$, let
\[
        m_b(F)=\ord_{x=b}F
\]
with $m_b(F)=0$ if $F(b)\ne0$.  Define
\[
        \kappa_b(F;f)=\ord_{z=0}F(\Gamma_b(z)),
\]
with $\kappa_b(F;f)=\infty$ if $F(\Gamma_b(z))\equiv0$.

\begin{theorem}\label{thm:contact-formula}
For $b\in\Zp$,
\[
        \kappa_b(F;f)=
        \begin{cases}
        0, & F(b)\ne0,\\[3pt]
        \infty, & F(b)=0,\ f(b)=b,\\[3pt]
        m_b(F), & F(b)=0,\ f(b)\ne b.
        \end{cases}
\]
More generally, if $a\in\Zp$, $\zeta\in\Zp$, and $b=\Gamma_a(\zeta)$, then the order of $F(\Gamma_a(z))$ at $z=\zeta$ is given by the same formula with this $b$.
\end{theorem}

\begin{proof}
If $F(b)\ne0$, the order is $0$.  If $F(b)=0$ and $f(b)=b$, the arc is constant and the pullback is identically zero.  If $F(b)=0$ and $f(b)\ne b$, write
\[
        F(x)=(x-b)^eU(x),
        \qquad e=m_b(F),
        \qquad U(b)\ne0.
\]
By \cref{thm:orbit-ball}, $\Gamma_b(z)-b$ has a simple zero at $z=0$, so $F(\Gamma_b(z))$ has order $e$.  The statement at $z=\zeta$ follows from the flow law.
\end{proof}

\subsection{Root-counting interpretation}

\begin{theorem}\label{thm:one-dimensional-root-count}
Let $F\in\Zp[x]$ be nonzero and let $a\in\Zp$.
\begin{enumerate}[label=\textup{(\alph*)}]
\item If $f(a)=a$, then $F(\Gamma_a(z))\equiv F(a)$.  Hence the return set is $\Np$ if $F(a)=0$ and is empty if $F(a)\ne0$.
\item If $f(a)\ne a$, set $\delta=f(a)-a$ and $s=\vp(\delta)$.  Then
\begin{equation}\label{eq:SI-one-dim-delta}
        \SI(F(\Gamma_a(z)))=
        \SI(F(a+\delta w))=
        \SI(F(a+p^s w)).
\end{equation}
Moreover,
\begin{equation}\label{eq:root-count-ball}
        \SI(F(\Gamma_a(z)))
        =\sum_{\substack{\alpha\in\OCp\\F(\alpha)=0\\
        \vp(\alpha-a)\ge s}}m_\alpha(F).
\end{equation}
Equivalently,
\begin{equation}\label{eq:derivative-formula}
        \SI(F(\Gamma_a(z)))=
        \max\left\{j:
        \vp\left(p^{sj}\frac{F^{(j)}(a)}{j!}\right)
        =
        \min_i\vp\left(p^{si}\frac{F^{(i)}(a)}{i!}\right)
        \right\}.
\end{equation}
\end{enumerate}
\end{theorem}

\begin{proof}
Part (a) follows from \cref{thm:orbit-ball}.  For (b), \cref{thm:orbit-ball} gives
\[
        \Gamma_a(z)=a+\delta W_a(z),
\]
where $W_a$ is an analytic automorphism of the closed unit disc.  Thus
\[
        F(\Gamma_a(z))=F(a+\delta W_a(z)).
\]
By \cref{fact:weierstrass}, analytic changes of parameter preserve the Strassmann index.  Since $\delta=p^su$ with $u\in\Zp^\times$, the change $w\mapsto uw$ is also an analytic automorphism, giving \eqref{eq:SI-one-dim-delta}.

The roots of $F(a+p^sw)$ with $w\in\OCp$ are precisely the roots $\alpha$ of $F$ with $\vp(\alpha-a)\ge s$, with multiplicities preserved.  The root-counting formula follows from \cref{fact:weierstrass}.  Finally,
\[
        F(a+p^sw)=\sum_{j\ge0}p^{sj}\frac{F^{(j)}(a)}{j!}w^j,
\]
which gives \eqref{eq:derivative-formula}.
\end{proof}

\begin{remark}\label{rem:roots-vs-returns}
The right-hand side of \eqref{eq:root-count-ball} counts all roots of $F$ in the closed $\Cp$-ball $a+p^s\OCp$.  Ordinary return times correspond only to those zeros of the interpolant whose $p$-adic time lies in the subset $\Np\subset\Zp$.  Thus the root count is an a priori zero bound; the actual return set can be smaller.
\end{remark}

\begin{corollary}\label{cor:weighted-return-bound}
Assume $f(a)\ne a$ and set $s=\vp(f(a)-a)$.  Then
\[
        \sum_{\substack{n\ge0\\F(f^n(a))=0}}
        m_{f^n(a)}(F)
        \le \SI(F(a+p^sw)).
\]
In particular, the number of ordinary returns is at most this Strassmann index.
\end{corollary}

\begin{proof}
Since $f(a)\ne a$, \cref{thm:orbit-ball} shows that $\Gamma_a$ is injective on
$\Zp$.  Hence no point $f^n(a)$ is fixed by $f$: otherwise
\[
        \Gamma_a(n+1)=f(f^n(a))=f^n(a)=\Gamma_a(n),
\]
contradicting injectivity.  By \cref{thm:contact-formula}, each ordinary return
therefore contributes exactly the multiplicity of the root it hits to the zero
divisor of $F(\Gamma_a(z))$.
By \cref{fact:strassmann,fact:weierstrass}, the total zero multiplicity in
$\Zp$ is bounded by the Strassmann index.
\end{proof}

\begin{corollary}\label{cor:practical-criteria}
Assume $f(a)\ne a$ and set $s=\vp(f(a)-a)$.
\begin{enumerate}[label=\textup{(\roman*)}]
\item If $\SI(F(a+p^sw))=0$, then there are no returns.
\item If $F(a)=0$ and $a$ is the only root of $F$ in $a+p^s\OCp$, counted with multiplicity $m_a(F)$, then the only return is $n=0$.
\item If already found return times account for total multiplicity $\SI(F(a+p^sw))$, then the search has found all returns.
\end{enumerate}
\end{corollary}

\subsection{Comparison with static \texorpdfstring{$p$}{p}-adic equations}\label{subsec:static-equation-comparison}

The one-dimensional certificate is related in spirit, but complementary, to classical solvability criteria for fixed $p$-adic equations.  Mukhamedov and Saburov studied the monomial equation $x^q=a$ over $\Qp$ \cite{MukhamedovSaburov2013}, while Mukhamedov, Omirov, and Saburov studied depressed cubic equations over $p$-adic fields, including solution counts in $\Zp^\times$, $\Zp$, and $\Qp$ \cite{MukhamedovOmirovSaburov2014}.  Those works concern static equations $F(x)=0$, where the unknown is $x$.  In contrast, the present framework fixes a local dynamical system and studies the time variable in
\[
        F(f^n(a))=0.
\]
After the orbit has been interpolated by $f^z(a)$, the Strassmann certificate gives an effective zero bound for the analytic function $F(f^z(a))$, and in dimension one this bound is exactly the Weierstrass root count of $F$ in the corresponding orbit ball.  Thus, for example, with $F(x)=x^3+Ax-B$, the method does not replace a cubic solvability criterion; it gives a local dynamical certificate for how often a chosen orbit can hit the roots of that cubic.  Similarly, the power-map application below studies equations such as $\alpha^{r^n}=\beta$, where the unknown is the time $n$ rather than the base variable of a fixed monomial equation.  The static solvability results are therefore not used in the proofs below, but they provide a useful comparison point for the orbit-intersection problem treated here.

\section{Arithmetic application: power maps on a \texorpdfstring{$p$}{p}-adic unit disc}\label{sec:power-map-application}

This section gives a concrete arithmetic-dynamical application of the one-dimensional certificate.  The point is not merely that the local theorem applies to a polynomial close to the identity, but that a familiar global map becomes such a polynomial after restricting to a residue disc and choosing a coordinate.  The resulting bound certifies intersections of an infinite exponential orbit with a finite target set.

Let $p\ge3$ and let $r\ge2$ be an integer satisfying
\[
        r\equiv1\pmod p.
\]
Consider the power map
\[
        \varphi_r:\mathbb G_m\longrightarrow\mathbb G_m,
        \qquad T\longmapsto T^r.
\]
The residue disc $1+p\Zp$ is invariant under $\varphi_r$.  In the coordinate
\[
        T=1+pX,
        \qquad X\in\Zp,
\]
the induced map is
\begin{equation}\label{eq:power-coordinate-map}
        g_r(X)=\frac{(1+pX)^r-1}{p}\in\Zp[X].
\end{equation}
Moreover
\begin{equation}\label{eq:gr-close-id}
        g_r(X)-X
        =(r-1)X+
        \sum_{k=2}^r \binom rk p^{k-1}X^k
        \in p\Zp[X].
\end{equation}
Thus $g_r(X)=X+p\Phi_r(X)$ for some $\Phi_r\in\Zp[X]$, and the one-dimensional theory of \cref{sec:one-dimensional} applies with $q=1$.

\begin{theorem}\label{thm:power-map-application}
Let $p\ge3$, let $r\ge2$ satisfy $r\equiv1\pmod p$, and let
\[
        \varphi_r(T)=T^r.
\]
Let $\alpha\in1+p\Zp$ be such that $\alpha^r\ne\alpha$, and put
\begin{equation}\label{eq:rho-power-map}
        \rho=\vp(\alpha^r-\alpha).
\end{equation}
Let $B$ be a finite subset of $1+p\Zp$.  Then
\begin{equation}\label{eq:power-map-bound}
        \#\{n\ge0:\alpha^{r^n}\in B\}
        \le
        \#\{\beta\in B:\vp(\beta-\alpha)\ge\rho\}.
\end{equation}
Furthermore, if a finite search has already found as many target hits as the right-hand side of \eqref{eq:power-map-bound}, then the search has found all target hits.
\end{theorem}

\begin{proof}
Write
\[
        a=\frac{\alpha-1}{p}\in\Zp.
\]
Then
\[
        g_r^n(a)=\frac{\varphi_r^n(\alpha)-1}{p}
        \qquad(n\ge0).
\]
For each $\beta\in B$, write
\[
        b_\beta=\frac{\beta-1}{p}\in\Zp,
\]
and define
\[
        F_B(X)=\prod_{\beta\in B}(X-b_\beta)\in\Zp[X].
\]
Then
\[
        F_B(g_r^n(a))=0
        \quad\Longleftrightarrow\quad
        \varphi_r^n(\alpha)\in B.
\]
Since $\alpha^r\ne\alpha$, we have $g_r(a)\ne a$.  The orbit-ball parameter for $g_r$ is
\[
        s=\vp(g_r(a)-a)
        =\vp\left(\frac{\alpha^r-\alpha}{p}\right)
        =\rho-1.
\]
By \cref{cor:weighted-return-bound}, the number of ordinary returns is bounded by the number of roots of $F_B$ in
\[
        a+p^s\OCp,
\]
counted with multiplicity.  Since $B$ is a set, these roots are simple.  A root $b_\beta$ lies in this ball precisely when
\[
        \vp(b_\beta-a)\ge s,
\]
which is equivalent to
\[
        \vp(\beta-\alpha)
        =1+\vp(b_\beta-a)
        \ge s+1=\rho.
\]
This proves \eqref{eq:power-map-bound}.  The final assertion follows because each found hit consumes one simple root in the Strassmann zero count, as in \cref{cor:practical-criteria}.
\end{proof}

\begin{corollary}\label{cor:power-map-avoidance}
Keep the notation of \cref{thm:power-map-application}.  Let $B\subset1+p\Zp$ be finite and let $e_\beta\ge1$ be integers.  Put
\[
        P_B(T)=\prod_{\beta\in B}(T-\beta)^{e_\beta}.
\]
Then the local Strassmann index of the interpolant $P_B(\varphi_r^z(\alpha))$ is
\[
        \sum_{\substack{\beta\in B\\ \vp(\beta-\alpha)\ge\rho}}e_\beta.
\]
Consequently:
\begin{enumerate}[label=\textup{(\roman*)}]
\item if $B\cap(\alpha+p^\rho\Zp)=\varnothing$, then $\alpha^{r^n}\notin B$ for every $n\ge0$;
\item only targets in the ball $\alpha+p^\rho\Zp$ can be hit;
\item once a finite search has found hits whose target multiplicities add up to the displayed sum, the search is certified complete.
\end{enumerate}
In particular, for a single target $\beta\in1+p\Zp$ satisfying $\vp(\beta-\alpha)<\rho$, the exponential equation
\[
        \alpha^{r^n}=\beta
\]
has no solution $n\in\Np$.
\end{corollary}

\begin{proof}
In the coordinate $T=1+pX$, write $b_\beta=(\beta-1)/p$.  Then $P_B(1+pX)$ is a nonzero scalar multiple of
\[
        \prod_{\beta\in B}(X-b_\beta)^{e_\beta}.
\]
Multiplication by a nonzero scalar does not change the Strassmann index.  As in the proof of \cref{thm:power-map-application}, the orbit ball in the $X$-coordinate is $a+p^{\rho-1}\Zp$, equivalently the orbit ball in the $T$-coordinate is $\alpha+p^\rho\Zp$.  The one-dimensional root-count formula therefore gives the displayed index.  The exclusion and completion criteria follow from \cref{cor:practical-criteria}.
\end{proof}

\begin{example}\label{ex:power-map-3-adic}
Let
\[
        p=3,
        \qquad r=4,
        \qquad \varphi(T)=T^4,
        \qquad \alpha=4.
\]
Then $\alpha\in1+3\mathbb Z_3$ and
\[
        \rho=\vp(4^4-4)=\vp(252)=2.
\]
Consider the finite target set
\[
        B=\{4,256,7\}\subset1+3\mathbb Z_3.
\]
The targets in the certified orbit ball are those congruent to $4$ modulo $9$:
\[
        4\equiv4\pmod9,
        \qquad 256\equiv4\pmod9,
        \qquad 7\not\equiv4\pmod9.
\]
Therefore \cref{thm:power-map-application} gives
\[
        \#\{n\ge0:4^{4^n}\in\{4,256,7\}\}\le2.
\]
But
\[
        \varphi^0(4)=4,
        \qquad
        \varphi^1(4)=4^4=256.
\]
The two allowed hits have already occurred, so the finite search is certified complete.  Hence
\[
        4^{4^n}\notin\{4,256,7\}
        \qquad(n\ge2),
\]
and in particular
\[
        4^{4^n}\ne7
        \qquad(n\ge0).
\]
This conclusion concerns an infinite exponential orbit, but the proof uses only the root count inside the single $3$-adic ball $4+9\mathbb Z_3$.
\end{example}

\section{Arithmetic application: avoiding roots of unity}\label{sec:roots-unity-application}

This section gives a second arithmetic-dynamical application of the one-dimensional contact theorem.  Here the target divisors are
\[
        F_N(x)=x^N-1,
        \qquad N\ge1,
\]
and the initial point is a torsion unit.  This may be viewed as the dynamical counterpart of the special monomial equations $x^N=1$: the question is not whether such roots exist over $\Qp$, but whether a tangent-to-identity $p$-adic orbit through a torsion unit can meet the divisor $x^N-1=0$ again.  The conclusion is uniform in $N$: although the degree of $F_N$ grows, the certified local Strassmann index is always either $0$ or $1$.

\begin{theorem}\label{thm:roots-unity-avoidance}
Assume $d=1$ and
\[
        f(x)=x+p^q\Phi(x),
        \qquad \Phi\in\Zp\langle x\rangle,
        \qquad q(p-1)>1.
\]
Let $\zeta\in\Zp$ be a root of unity and assume that $f(\zeta)\ne\zeta$.  Then, for every integer $N\ge1$,
\begin{equation}\label{eq:roots-unity-return-set}
        \{n\in\Np:f^n(\zeta)^N=1\}
        =
        \begin{cases}
        \{0\}, & \zeta^N=1,\\[2pt]
        \varnothing, & \zeta^N\ne1.
        \end{cases}
\end{equation}
Moreover, if $\Gamma_\zeta(z)=f^z(\zeta)$, then
\begin{equation}\label{eq:roots-unity-SI}
        \SI\bigl(\Gamma_\zeta(z)^N-1\bigr)
        =
        \begin{cases}
        1, & \zeta^N=1,\\[2pt]
        0, & \zeta^N\ne1.
        \end{cases}
\end{equation}
Thus the local zero bound is sharp and independent of $N$.
\end{theorem}

\begin{proof}
Set
\[
        s=\vp(f(\zeta)-\zeta).
\]
Since $f(x)\equiv x\pmod{p^q}$ and $f(\zeta)\ne\zeta$, we have $s\ge q$.  The hypothesis $q(p-1)>1$ gives
\[
        s>\frac1{p-1}.
\]
By \cref{thm:one-dimensional-root-count}, the Strassmann index of $\Gamma_\zeta(z)^N-1$ is the number of roots of $x^N-1$ in the closed ball
\[
        \zeta+p^s\OCp,
\]
counted with multiplicity.

Let $\alpha$ be a root of $x^N-1$ in this ball.  If $\alpha\ne\zeta$, then $\alpha/\zeta$ is a nontrivial root of unity.  By \cref{fact:roots-unity-separation},
\[
        \vp(\alpha-\zeta)
        =\vp\bigl(\zeta(\alpha/\zeta-1)\bigr)
        \le \frac1{p-1}<s,
\]
which contradicts $\alpha\in\zeta+p^s\OCp$.  Hence the ball contains no root of $x^N-1$ except possibly $\zeta$ itself.

If $\zeta^N\ne1$, it contains no root, so the index is $0$ and there are no returns.  If $\zeta^N=1$, it contains exactly the root $\zeta$.  This root has multiplicity one because $x^N-1$ is separable over the characteristic-zero field $\Cp$.  Therefore the index is $1$.  Since $n=0$ is already a return in this case, \cref{cor:weighted-return-bound} shows that there can be no further ordinary return time.
\end{proof}

\begin{corollary}\label{cor:polynomial-family-roots-unity}
Let $q(p-1)>1$ and let
\[
        f_u(x)=x+p^q u x^r,
        \qquad u\in\Zp^\times,
        \qquad r\ge1.
\]
Let $\zeta\in\Zp^\times$ be a root of unity.  Then $f_u(\zeta)\ne\zeta$, and for every $N\ge1$,
\[
        \{n\in\Np:f_u^n(\zeta)^N=1\}
        =
        \begin{cases}
        \{0\}, & \zeta^N=1,\\[2pt]
        \varnothing, & \zeta^N\ne1.
        \end{cases}
\]
In particular, for $p\ge3$ and $f(x)=x+px^2$, one has
\[
        (f^n(1))^N=1
        \quad\Longleftrightarrow\quad n=0
\]
for every $N\ge1$.
\end{corollary}

\begin{proof}
For a root of unity $\zeta\in\Zp^\times$,
\[
        f_u(\zeta)-\zeta=p^q u\zeta^r\ne0.
\]
Thus \cref{thm:roots-unity-avoidance} applies.  The final assertion is the special case $q=1$, $u=1$, $r=2$, and $\zeta=1$.
\end{proof}

\begin{remark}\label{rem:roots-unity-effective}
For the divisors $F_N(x)=x^N-1$, the degree grows with $N$.  A naive degree bound would give at best $N$ possible local zeros.  The orbit-ball certificate gives the exact Strassmann index, which is either $0$ or $1$, uniformly in $N$.  Thus the local method proves a sharp return statement for an infinite arithmetic family of target divisors using only the $p$-adic separation of roots of unity and the one-dimensional contact theorem.
\end{remark}

\section{Worked examples}\label{sec:examples}

\begin{example}
Let
\[
        p=3,
        \qquad f(x)=x+3x^2,
        \qquad a=1,
        \qquad F(x)=x-1.
\]
Here $F(a)=0$, $\Phi(x)=x^2$, and
\[
        F'(1)\Phi(1)=1\in\mathbb Z_3^\times.
\]
By \cref{prop:first-order-escape}, $\SI(H_F)=1$, so $f^n(1)\ne1$ for all $n\ge1$.  In the one-dimensional interpretation, $s=\vp(f(1)-1)=1$ and
\[
        F(1+3w)=3w,
\]
whose Strassmann index is $1$.
\end{example}

\begin{example}
Let
\[
        p=3,
        \qquad f(x)=x+3x^2,
        \qquad a=1,
        \qquad F(x)=(x-1)(x-4).
\]
Since $s=\vp(f(1)-1)=1$,
\[
        F(1+3w)=(3w)(3w-3)=9w(w-1).
\]
Thus $\SI(F(f^z(1)))=2$.  The two roots in the orbit ball are $1$ and $4$, and indeed $f^0(1)=1$ and $f^1(1)=4$.  Their total multiplicity is $2$, so there are no further returns to $F=0$.
\end{example}

\begin{example}
Let
\[
        p=3,
        \qquad f(x)=x+9,
        \qquad a=0.
\]
Then $f^z(0)=9z$ and the orbit ball is $9\mathbb Z_3$.  If $F(x)=x-3$, then
\[
        F(9w)=9w-3=3(3w-1),
\]
whose Strassmann index is $0$; hence there is no return.  If $F(x)=x-18$, then
\[
        F(9w)=9(w-2),
\]
so the unique return occurs at $z=2$.
\end{example}

\begin{example}
Let
\[
        p=3,
        \qquad f(x)=x+3,
        \qquad a=0,
        \qquad F(x)=x(x-3).
\]
Then $f^n(0)=3n$ and
\[
        H_F(z)=F(3z)=9z(z-1)=9z^2-9z,
\]
so $\SI(H_F)=2$.  The orbit values are
\[
        y_0=0,
        \qquad y_1=0,
        \qquad y_2=18,
        \qquad y_3=54.
\]
For $M=1$, the truncated coefficients are zero and no certificate is obtained.  For $M=2$,
\[
        B_2=18,
        \qquad C_1^{(2)}=-9,
        \qquad C_2^{(2)}=9,
\]
so $\alpha_2=2$.  But
\[
        \lambda_{3,1}(2)=2,
\]
and the strict inequality required by \cref{thm:certificate} fails.  At $M=3$, one has $B_3=0$ and
\[
        \lambda_{3,1}(3)=3.
\]
Thus $\alpha_3=2<3$, and the certificate succeeds, returning $\SI(H_F)=2$.
\end{example}

\begin{example}\label{ex:zooming-certificate}
Continue with
\[
        p=3,
        \qquad f(x)=x+3,
        \qquad a=0,
        \qquad F(x)=x(x-3).
\]
The unsplit interpolant is
\[
        H_F(z)=F(3z)=9z(z-1).
\]
At truncation $M=2$ the coefficient minimum is $\alpha_2=2$, while
\[
        \lambda_{3,1}(2)=2,
\]
so the strict inequality in \cref{thm:certificate} does not yet certify the
index.  Now split the time variable modulo $3$.  Here $h=1$, the local map is
$f^3(x)=x+9$, and the certificate tails use $q+h=2$, for which
\[
        \lambda_{3,2}(1)=4.
\]
The three zoomed interpolants are
\[
\begin{aligned}
        H_{F;0,1}(u)&=F(9u)=27u(3u-1),\\
        H_{F;1,1}(u)&=F(3+9u)=27u(1+3u),\\
        H_{F;2,1}(u)&=F(6+9u)=18+81u+81u^2.
\end{aligned}
\]
For $b=0$ and $b=1$, the $M=1$ truncation has coefficient minimum $3<4$ and
returns $\SI(H_{F;b,1})=1$.  For $b=2$, the same truncation gives
\[
        C_0^{(1)}=18,
        \qquad C_1^{(1)}=162,
\]
so the coefficient minimum is $2<4$ and the certificate returns
$\SI(H_{F;2,1})=0$.  Thus the first residue-class refinement certifies the two
actual hits $n=0$ and $n=1$, and certifies avoidance on the class
$n\equiv2\pmod3$, without waiting for the unsplit $M=3$ certificate.
\end{example}

\begin{example}\label{ex:arc-gcd-gain}
Let
\[
        p=3,
        \qquad f(x)=x+3,
        \qquad a=0,
        \qquad \Gamma_a(z)=3z.
\]
Let $V$ be the target in the affine line cut out by the two equations
\[
        F_1(x)=x(x-3),
        \qquad
        F_2(x)=x(x-6).
\]
The $\mathbb Z_3$-points satisfying both equations consist only of $x=0$.
Along the interpolated orbit,
\[
        H_1(z)=F_1(3z)=9z(z-1),
        \qquad
        H_2(z)=F_2(3z)=9z(z-2).
\]
Each single equation has Strassmann index $2$.  However, the ideal
$(H_1,H_2)\subset\mathbb Q_3\langle z\rangle$ is generated, up to a unit, by $z$:
indeed $H_1$ and $H_2$ have the only common zero $z=0$ in the closed unit disc.
Thus the arc-gcd certificate gives
\[
        G_{\Gamma_a,V}(z)=z,
        \qquad \SI(G_{\Gamma_a,V})=1.
\]
The gcd certificate is therefore strictly sharper than either hypersurface
certificate: $F_1$ alone permits the two times $n=0,1$, $F_2$ alone permits the
two times $n=0,2$, while the target cut out by both equations is hit only at
$n=0$.
\end{example}

\section{Global-to-local role and finite extensions}\label{sec:global}

The results above are local.  In the $p$-adic DML method, one starts with a global algebraic dynamical system, chooses a prime of good reduction, and studies an integral orbit.  Since the residue field is finite, the reduced orbit is eventually periodic.  After passing to a residue class of the time variable, one obtains a subsequence
\[
        f^{rn+b}(a)
\]
that remains in a single residue polydisc.  In favorable settings, notably for \'{e}tale maps after replacing the map by a further iterate, the local map on that polydisc is congruent to the identity.  In coordinates, the local problem has the form
\[
        x\mapsto x+p^q\Phi(x),
        \qquad q(p-1)>1,
\]
and the present certificates apply to the analytic function interpolating
\[
        F(f^{rn+b}(a)).
\]
This is the same local analytic skeleton used in the Bell--Ghioca--Tucker approach \cite{BGT2010}; Poonen's theorem gives a clean sufficient condition for the interpolation step \cite{Poonen2014}.

\subsection{Finite extensions}

Let $K/\Qp$ be a finite extension, let $\OK$ be its valuation ring, and normalize the valuation by $\vp(p)=1$.  Put
\[
        \A_K=\OK\langle x_1,\ldots,x_d\rangle.
\]
If
\[
        f(x)=x+p^q\Phi(x),
        \qquad \Phi\in\A_K^d,
        \qquad q(p-1)>1,
\]
then the exact certificate, finite-precision certificate, sharper-tail criterion, one-shot bound, and first-order escape criterion remain valid over $K$.  The proofs are identical: $\Delta$ maps $\A_K$ into $p^q\A_K$, the denominators are controlled by $\vp(m!)$, and Strassmann's theorem holds over complete discretely valued extensions of $\Qp$.

\begin{remark}
The coefficient certificate remains valid over $K$, but the orbit-ball interpretation in \cref{sec:one-dimensional} uses the one-dimensional $\Qp$-analytic inverse theorem in the coordinate under consideration.  Over a proper finite extension $K$, a $\Zp$-analytic arc in $\OK$ need not fill an entire $\OK$-ball.  Thus the simple formula as a count of roots in the whole $\OK$-ball requires additional hypotheses or a separate $K$-analytic parametrization.
\end{remark}

\section{Limitations and possible extensions}\label{sec:limitations}

The method requires analytic interpolation of the iterates.  The congruence
\[
        f\equiv\id\pmod{p^q},
        \qquad q(p-1)>1,
\]
is a convenient sufficient condition.  If a global system cannot be reduced to this local form after passing to residue classes and iterates, the certificates do not apply directly.

The main obstruction is ramified or non-\'{e}tale behavior.  Contracting or singular directions may fail to admit a single analytic interpolation in the time variable in the simple Mahler form used here.  Such phenomena are part of the broader difficulty of DML beyond the \'{e}tale case; see the discussion in Xie \cite{Xie2023}.  Extending the effective Strassmann certificates to controlled non-\'{e}tale models would require new local normal forms or a separate treatment of contracting directions.

There is also a geometric limitation in higher dimension.  In dimension one, a non-fixed interpolated orbit parametrizes a ball, so the index becomes an ordinary root count.  In higher dimension, the interpolated orbit is a one-dimensional analytic arc in a polydisc, and the corresponding invariant is an intersection multiplicity with that arc.  The coefficient certificate still computes the Strassmann index, but there is no equally simple ambient ball-counting interpretation.

The arc-ideal perspective suggests a concrete computational extension: develop
certified finite-precision Weierstrass-gcd routines for the functions
$F_i(\Gamma_a(z))$ arising from several defining equations.  The present paper
uses the arc-gcd only after a generator has been obtained or certified; it does
not supply such a gcd routine.  A future routine of this kind would turn the
conceptual generator $G_{\Gamma_a,V}$ of \cref{sec:arc-ideals} into a practical
multiequation stopping rule.  A second direction is to combine
residue-class zooming with local normal forms for mildly non-\'{e}tale maps, so
that contracting directions are controlled by valuation growth while the
\'{e}tale directions are treated by Strassmann certificates.  These extensions
are not needed for the results proved here, but they are natural next steps for
making the local certificate useful in broader DML computations.

\section{Conclusion}

The standard $p$-adic interpolation method reduces a local piece of the Dynamical Mordell--Lang problem to the zero set of an analytic function $H_F:\Zp\to\Qp$.  Under the congruence hypothesis $f\equiv\id\pmod{p^q}$ with $q(p-1)>1$, the Mahler coefficients of $H_F$ have strong divisibility, and the conversion from Mahler polynomials to ordinary powers has explicitly controlled denominators.  This yields a finite certificate for the exact Strassmann index whenever $H_F$ is nonzero, as well as finite-precision, refined-tail, one-shot, and first-order variants.

The revised framework adds two further layers.  First, residue-class zooming turns the subclass $b+p^h\Zp$ of times into a new local system governed by $f^{p^h}$, which is $p^{q+h}$-close to the identity and therefore has stronger certificate tails.  Second, the arc-ideal viewpoint replaces separate hypersurface bounds by the gcd of all defining equations along the interpolated orbit, giving the natural intersection-multiplicity bound for multiequation targets.

In dimension one, the same local structure has a direct geometric interpretation: non-fixed interpolated orbits parametrize their $p$-adic orbit balls, contact with $F=0$ is root multiplicity, and the Strassmann index is the Weierstrass root count of $F$ in that ball, hence an upper bound for ordinary return times.  The power-map and root-of-unity applications show how this local description yields explicit certified bounds for intersections of infinite arithmetic orbits with finite target sets and torsion divisors.  The resulting tools are local and effective rather than global, but they make the local zero-bound step in the $p$-adic DML method explicit and certifiable.

\section*{Data and code availability} No data were used in the present work.

\section*{Declaration of competing interest}
The author declares that he has no known competing financial interests or personal relationships that could influence the work reported in this study.

\section*{Funding}
This research received no specific grant from any funding agency in the public, commercial, or not-for-profit sectors.

\end{document}